\documentclass[a4paper,11pt]{article}

\usepackage{a4,latexsym,graphicx,times,float} 
\usepackage{amssymb} 
\usepackage{amsmath}
\usepackage[T1]{fontenc}
\usepackage{enumitem}
\usepackage[parfill]{parskip}
\usepackage{wrapfig}
\usepackage{caption}
\usepackage{amsthm}
\usepackage{fullpage}
\usepackage{color}

\newtheorem{theorem}{Theorem}
\newtheorem{lemma}{Lemma}

\begin{document}

\title{A re-entrant phase transition in the survival of secondary infections on networks}
\author{Sam Moore$^1$ \and Peter M\"{o}rters$^2$ \and Tim Rogers$^1$}

\date{
$^1$Department of Mathematical Sciences, University of Bath, Bath, BA2 7AY, UK \\
$^2$Mathematisches Institut, Universit\"{a}t zu K\"{o}ln, Weyertal 86-90, 50931 K\"{o}ln, Germany
}

\maketitle

\begin{abstract}
We study the dynamics of secondary infections on networks, in which only the individuals currently carrying a certain primary infection are susceptible to the secondary infection. In the limit of large sparse networks, the model is mapped to a branching process spreading in a random time-sensitive environment, determined by the dynamics of the underlying primary infection. When both epidemics follow the Susceptible-Infective-Recovered model, we show that in order to survive, it is necessary for the secondary infection to evolve on a timescale that is closely matched to that of the primary infection on which it depends.  
\end{abstract}

\section{Introduction}
Superinfections are a major cause of global mortality and morbidity. For example, the WHO estimates 15 million cases worldwide of Hepatitis D, which spreads only amongst carriers of Hepatitis B and greatly worsens their prognosis \cite{WHO2017}. There is a need, therefore, to develop a robust understanding of the conditions under which outbreaks of secondary infections are possible. Coevolving infections have been studied previously in the case of symbiotic/antagonistic relationships where infections mutually affect fitness \cite{Grassberger2016,Chen2013,Cai2015a}, however, relatively little is known in the case that one infection has a strict obligate relationship with another. 

In 2013, Court, Blythe and Allen \cite{Court2013} introduced a model of hierarchical infection referred to as the \emph{stacked contact process}. Their model concerns the fate of a population of coevolving hosts, spreading as a contact process on a lattice, and parasites, spreading as a contact process restricted to sites currently occupied by hosts. In epidemiological language, the contact processes of \cite{Court2013} correspond to coupled Susceptible-Infective-Susceptible (SIS) epidemics; empty lattice sites are interpreted as susceptible individuals, who may be infected by the primary (host) and then secondary (parasite) infections. Simulations of this model system revealed a surprising feature: the success of the parasites depends non-monotonically on the turnover rate of the host population. Specifically, for the parasite to succeed, it is necessary for the dynamics of the host population to be neither too fast, nor too slow. Later in \cite{Lanchier2014}, Lanchier and Zhang rigorously established the main features of the phase diagram for the stacked contact process. 

\begin{figure}
\centering
\includegraphics[width=0.55\linewidth]{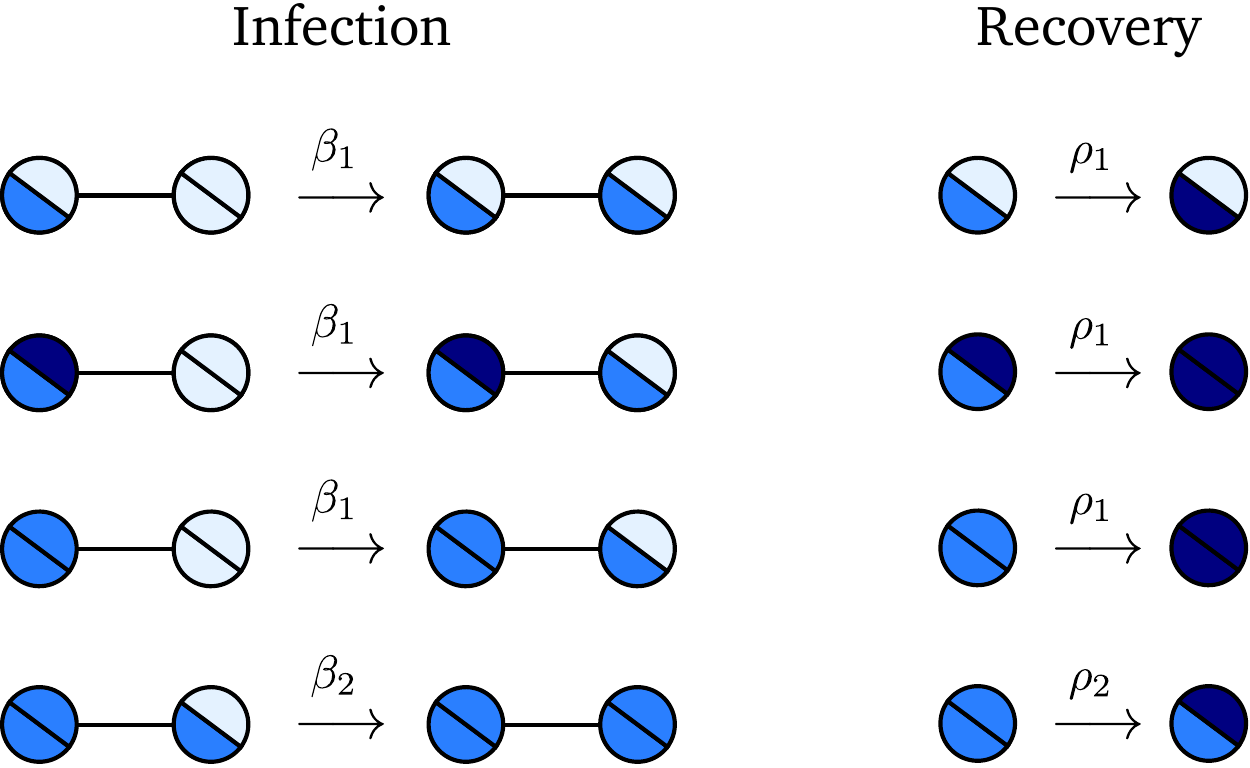}
\caption{Possible events and their rates in the network superinfection model. Circles represent nodes in the network, with the state of the primary (resp. secondary) infection shown by the colour of the lower-left (resp. top-right) sector; light denotes susceptible, midtone denotes infective, dark denotes recovered. }
\label{cartoon}
\end{figure}

At around the same time, Newman and Ferrario \cite{Newman2013} independently proposed a related model in the context of epidemic dynamics in social contact networks. They considered a pair of Susceptible-Infective-Recovered (SIR) epidemics with a strictly obligate relationship such that the secondary infection is only transmitted amongst those who have \emph{recovered} from the primary. In this formulation, the dynamics of the two diseases are completely separated in time, allowing for analytical treatment of the model using ``cavity method'' techniques which have been quite successful in the study of epidemics on networks (see, e.g. \cite{Newman2002,Rogers2015}). The introduction of network structure to the population in \cite{Newman2013} has the advantage of improving the relevance of the model for human epidemic dynamics, however, by separating the dynamics of the two diseases this model cannot display the curious interaction between infection timescales observed in \cite{Court2013,Lanchier2014}.

In this paper we study the dynamics of coevolving SIR superinfections in sparse contact networks. We consider a population of individuals occupying the vertices of an Erd\H{o}s-R\'enyi (ER) random graph with mean degree $c$. A primary infection spreads through the population with infective individuals passing the disease on to their neighbours with rate $\beta_1$, and recovering from the disease with rate $\rho_1$. Individuals who are carrying a \emph{live} primary infection may also play host to a secondary infection, which spreads and recovers with rates $\beta_2,\rho_2$ respectively. See Fig.~\ref{cartoon} for an illustration of the possible state transitions.  \textcolor{black}{As in \cite{Newman2013}, our secondary infection is restricted to spread on the subgraph of hosts infected with the primary, however, differently from that paper we consider the more complex case in which this subgraph is evolving in time due to the recovery of primary infections. }

As well as arguably improving the realism of the model, moving from lattice to network topologies allows us access to a rigorous branching process approximation --- an approach that has previously enjoyed success in approximating SIR-type models in large populations, as seen for instance in \cite{Bartoszynski,SarabjeetSingh2014}. By coupling the dynamics of the secondary infection to those of a multi-type branching process, we will be able to characterise the phase diagram of the system. Note that, in the context of large finite networks, when we discuss \textit{survival} of the infection we mean an asymptotically positive proportion of vertices become infected at some  point  in time.
\begin{figure}[ht]
\centering
\includegraphics[width=0.55\linewidth, trim=120 290 120 240]{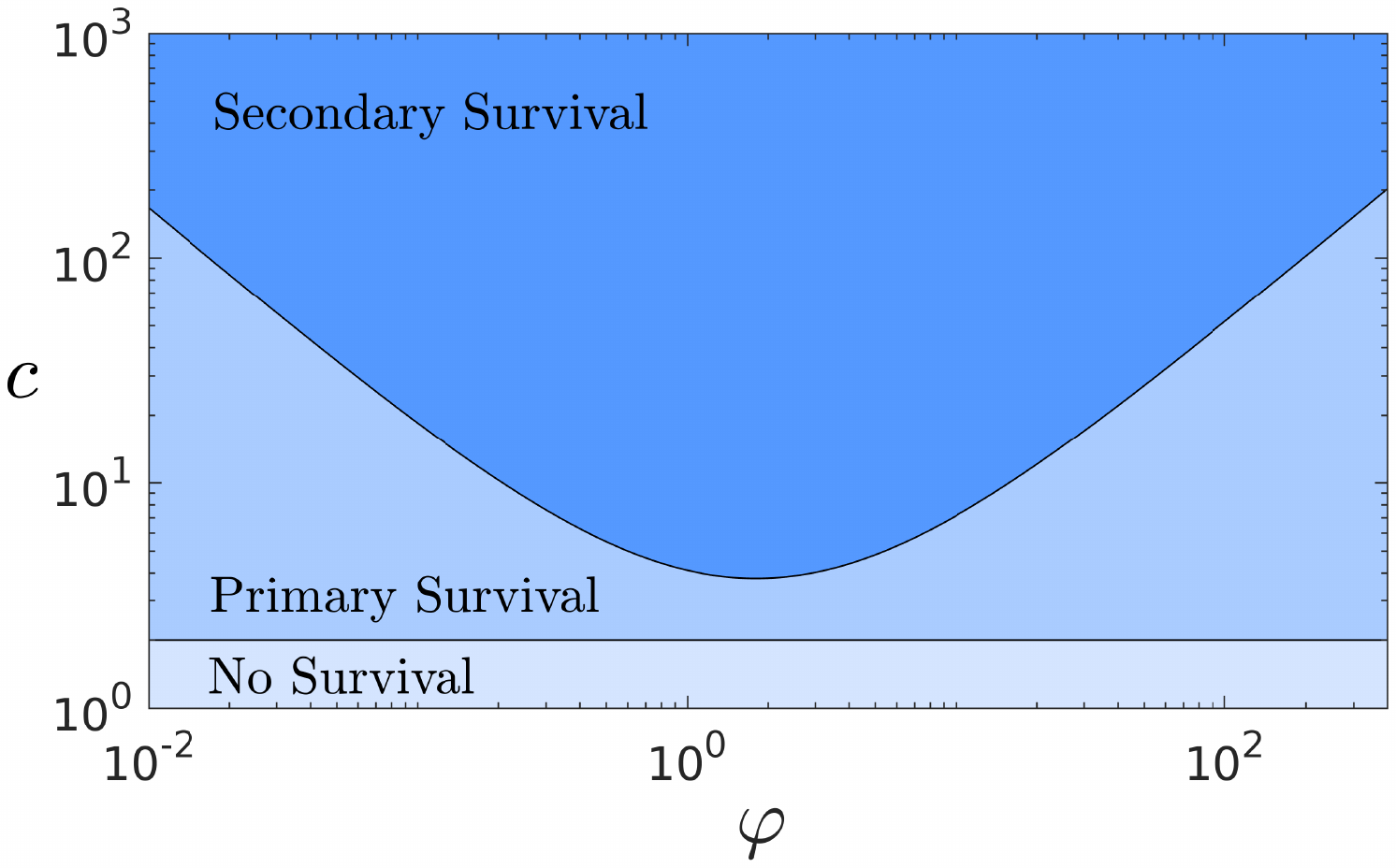}
\caption{Phase diagram of the superinfection network model for fixed $\beta_1/\rho_1=\beta_2/\rho_2$, shown as a function of the relative timescale $\varphi=\beta_1/\beta_2$ of the infections and the connectivity $c$ of the network. The secondary infection survives with positive probability only in a convex region whose boundary is characterised in our Theorem \ref{th1}. In this log-log plot, the asymptotic slope of the boundary is $-1$ for small $\varphi$, and $1$ for large $\varphi$, as implied by the scaling laws in \eqref{SL}.}
\label{fig:regions}
\end{figure}

The success of the primary infection is controlled by the connectivity (mean node degree) $c$ of the network, and the ratio of the infection and recovery rates $\alpha:=\beta_1/\rho_1$. This parameter is well understood as the basis of the classical single infection process; for fixed $\alpha$ there exists a critical value of $c$ above which the infection survives with positive probability and at or below which we have certain extinction, see e.g. \cite{Henkel2008}. Note that simultaneously adjusting $\beta_1$ and $\rho_1$ by a multiplicative factor will change the timescale of the disease dynamics, but will not alter the probability of survival since $\alpha$ is unchanged. 

Inspired by the results of \cite{Court2013,Lanchier2014}, we are interested here in the behaviour possible when the primary and secondary diseases are similarly virulent, but may differ in their in timescales. To this end, we \textcolor{black}{will mainly} concentrate on the case that $\beta_2/\rho_2=\alpha$ also. \textcolor{black}{We have made this choice only for simplicity of presentation; the more general case is in fact covered by Lemma 2, and the results are not qualitatively different in other cases.} Three parameters then describe success of secondary infection: the connectivity of the underlying graph, $ c $; the ratio between rates of spread and recovery,  $ \alpha  $; and, crucially, the relative timescales of the two infections, $ \varphi:=\beta_1/\beta_2  $. If $\varphi\gg1$ then the dynamics of the primary infection are very much faster than those of the secondary; if $\varphi\ll1$ then they are much slower. 

In order for the secondary infection to survive it is perhaps intuitive that it must progress at a rate fast enough compared to the primary infection, else the primary infection will have itself recovered and subsequently ended the secondary infection before it has a chance to spread. Perhaps more surprisingly however we shall also show that the secondary infection should not act too quickly as this too compromises survival potential. Our characterisation of the survival of the secondary infection is illustrated in Fig.~\ref{fig:regions} and summarised by our main result:

\begin{theorem}\label{th1}
For all $\alpha,\varphi>0$ there exists a critical connectivity $c^\star$ such that, in the limit of large network size, for $c<c^\star$ the secondary infection dies out with probability one, and for $c>c^\star$ it survives with positive probability. 

Furthermore, the critical connectivity $c^\star$ is found to be the smallest positive solution of the implicit equation 
\begin{equation}
c^\star = \frac{(1+\alpha+\alpha\varphi)(1+\varphi+2\alpha+\alpha\varphi)}{\varphi} \frac{_0F_1 (2+\varphi+(1+\varphi)/\alpha;-c^\star\varphi/\alpha^2)}{_0F_1 (3+\varphi+(1+\varphi)/\alpha;-c^\star\varphi/\alpha^2)}\,,
\label{eig0}
\end{equation}
where $_0F_1(a;z)=\sum_k\textcolor{black}{\frac{z^k}{k!(a)_k}}$ is a hypergeometric function. In particular, for large and small $\varphi$ we have the scaling behaviour
\begin{equation}
c^\star =\mathrm{\Theta}(\varphi)\quad\text{for }\,\,\varphi\to\infty\,,\quad c^\star =\mathrm{\Theta}(1/\varphi)\quad\text{for }\,\,\varphi\to0\,.
\label{SL}
\end{equation}
\end{theorem}

\textcolor{black}{Here we have made use of ``big theta'' notation, defined as follows: $f(x)=\mathrm{\Theta}(g(x))$ as $x\to \infty$ (resp. $x\to 0$) if there exist positive constants $L$ and $U$ and $X$ such that $\forall x>X$ (resp. $x<X$) we have $$L g(x)< f(x) <U g(x)\,.$$}

The remainder of the article is organised as follows: in the next section we map the early dynamics of our network superinfection model to a certain multitype branching process; in Section 3 we compute the long-time behaviour of this process and thus give the proof of Theorem \ref{th1}; Section 4 is for discussion, including illustrative numerical results.  

\section{Branching process description}

\subsection{Primary infection}

We begin by recapping the standard branching process approximation to the dynamics of an infection spreading on an Erd\H{o}s-R\'{e}nyi random graph \cite{Foundation1956,Bartoszynski}. Heuristically, the method relies on the fact that for fixed connectivity, short cycles become asymptotically rare in the limit of large graphs, meaning that during the crucial early dynamics of the infection, each susceptible node may have at most one infective neighbour. 

Let us consider the infection spread  as generational; the $ n^{th} $ generation being the individuals at graph distance $ n $ from the seed vertex that gain the infection at any point in time. In this way the primary infection is modelled as a simple Galton Watson process described by the quantity $ Z_{n } $, giving the number of individuals in the $ n^{th } $ generation. The \textit{offspring distribution} describes the probability $p_i$ for an individual to pass the infection on to $i$ others in the next generation. If the offspring distribution has mean $\mu$, the expected number of infected individuals at distance $n$ from the seed is then given by $\mathbb{E}Z_{n}=\mu^n$. If $\mu\leq1$ then the branching process will almost surely go extinct after finitely many generations; if $\mu>1$ then it may survive; survival of the branching model being simply characterised by the size of the $ n^{th} $ generation being non zero for all $ n $. 

In the network SIR model, the number of offspring is equated with the number of neighbours (other than the single infected `parent') that an infective node succeeds in transmitting the disease to before it recovers. There are several sources of randomness: the number of neighbours to potentially infect, the recovery time, and times of infection. We note that whilst the fates of the neighbours of an infected node are not independent (they are jointly exposed to the random time to recovery of the parent) the mean of the offspring distribution can be found simply by multiplying the probability $\alpha/(1+\alpha)$ to infect any given neighbour before recovery, with the expected number of neighbours to infect, $c$. From standard branching process theory, we thus deduce that in the limit of large networks the primary infection will have a non-zero chance of survival if and only if $\mu>1$, that is, if $c>1+1/\alpha$. 

For finite graphs, the coupling between the random graph and branching process model is of course only local. Suppose in a population of size $ N $, in generation $  n $ we have $ m  $ infected individuals, so $ Z_{n }= m  $ in the branching model. We then have errors coming from the fact that each infective may only be connected to at most $ N-m  $ susceptibles (not constant for each generation) as well as the fact that each of these may not be unique (and so children in the subsequent generation of the branching process may not be unique). However when $ m = o\sqrt{n}$ the random graph may be coupled to the branching model with high probability; for a proof of this see \cite{Durrett}. 

\subsection{Secondary infection}
\label{branch2}
\textcolor{black}{For the primary infection, to determine if the probability of survival is positive only requires} knowledge of two quantities: the expected number of susceptible neighbours an individual has, and the chance any one of those will gain the infection. The difficulty with modelling the secondary infection is that the first of these is dynamic, since the subgraph composed of individuals currently carrying the primary infection changes with time. We account for this additional complexity by introducing a \emph{type} parameter $t$, which specifies the time elapsed between the primary and secondary infections. Specifically, if an individual acquires the primary infection at time $t_1$ and the secondary at time $t_2$, then they are said to have type $t=t_2-t_1$. 

\textcolor{black}{It is clear that at least this much information is required to predict the potential of an individual to transmit the secondary infection to new hosts; for example, the larger $t$, the more likely an individual is to pass on the primary infection long before it passes on the secondary, by which time the primary infection in the new host may have recovered. We will see in Section 3.1 that in fact knowledge of $t$ is enough to completely characterise the distribution of the number and timing of new secondary infections arising from an individual.} The progress of the secondary infection is then mapped to that of a multi-type branching process with type space $\mathbb{T} =[0,\infty)$. 

Where previously survival was predicted by just the mean number of offspring, now the picture is more complicated, and we are required to compute the \emph{intensity of production} of all types of offspring resulting from all types of parents. This information is captured in the kernel $\mu(t'|t)$, which is defined by the property that the expected number offspring with types in the interval $[a,b]$ coming from a parent of type $t$ is given by the integral of $\mu(t'|t)$ over $t'\in [a,b]$. This kernel defines a linear operator with the action
\begin{equation}
M[\psi](t')=\int\mu(t'|t)\psi(t)\text{d}t\,.
\label{Mdef}
\end{equation}
In words, $M[\psi]$ describes the expected size and composition of the population of offspring arising from a population of parents with types given by $\psi$. 

We say that a kernel $\mu$ defined over an interval $I$ is: strictly positive if $\forall t,t'\in I$ we have $\mu(t'|t)>0$; uniformly positive if $\exists\, \varepsilon>0$ such that $\forall t,t'\in I, \mu(t'|t)>\varepsilon$; integrable if $\iint \mu(t'|t)\textrm{d} t\,\textrm{d} t'<\infty$. We assume that $M$ can be defined as a linear operator $M\colon \mathcal{C}_{\text{b}}(\bar{\mathbb{T}}) \to \mathcal{C}_{\text{b}}(\bar{\mathbb{T}})$ over the space of continuous bounded functions on the compact interval $\bar{\mathbb T}=[0,\infty]$ equipped with the supremum norm, and in particular that $\mu$ has vanishing mass as $t$ goes to infinity. One then has the following general result:
\begin{lemma}\label{Z}
Let $\{Z_n\}$ be a multi-type branching process on $\mathbb{T} =[0,\infty)$ with production operator $M$ arising 
as above from a kernel $\mu$ that is strictly positive on $\mathbb{T}$, integrable, and continuous in both arguments, then 
\begin{enumerate}
\item There exists an eigenvalue $\lambda >0$ equal to the spectral radius of $M$, moreover, this is the only eigenvalue corresponding to a non-negative eigenfunction   
\item If $ \lambda < 1 $ then the process goes extinct in finite time with probability one
\item If $ \lambda > 1 $ then the process survives with positive probability.
\end{enumerate}
\end{lemma}
\begin{proof}

\begin{enumerate}
\item[]
\item For the first part, we observe that the properties of $\mu$ imply the compactness of $M$ on $ \mathcal{C}_{\text{b}}(\bar{\mathbb{T}})$ by virtue of the Arz\'ela-Ascoli theorem~\cite[IV.6.7]{DS}.
The Krein-Rutman theorem \cite[Th 1.3 $\S$3.2]{pinsky1995positive} then gives that the spectral radius is a positive eigenvalue and by  \cite[Theorem 7.3]{Anselone1974} the only nonzero eigenvalue with a non-negative eigenfunction.

\item We simply observe that if $ \lambda < 1 $ then $\|M^n[\psi]\|\to0$ for all $\psi$, hence we have convergence of the expected generation size to zero (i.e. $\mathbb{E}Z_n\to0$), which implies extinction in finite time with probability one.

\item We make use of results of Harris \cite[ $ \S $3]{Harris1989} who proved positive survival probability for multi-type branching processes with a uniformly positive kernel. Our kernel $\mu$ is not uniformly positive, but we are able to couple to such a process by restricting to a bounded type space $[0,T]$. Choosing $T$ large enough forces close agreement in the maximum eigenvalues of the corresponding production operators.

Let us start by considering the process \smash{$\{Z_n^{_{(T)}}\}$} obtained from $\{Z_n\}$ by removing all
individuals of type greater than $T$ along with their descendants. The law of \smash{$\{Z_n^{_{(T)}}\}$} is that of a multitype branching process on $[0,T]$ with operator 
$  M^{(T)} :\mathcal{C}_{\text{b}}[0,T] \to \mathcal{C}_{\text{b}}[0,T]  $
defined by 
\begin{align}
M^{(T)}[\psi](t')=\int_{[0,T]}\mu(t'|t)\psi(t)\text{d}t\,.
\end{align}
Note that \smash{$\inf_{t,t'\in[0,T]}\mu(t'|t)>0$} and so the kernel is strictly positive and we refer to  
\cite[$ \S $3]{Harris1989}  to prove both the existence of a positive top eigenvalue $\lambda^{{(T)}}$ of $ M^{(T)} $ strictly greater in magnitude than all others and survival of the process  \smash{$\{Z_n^{_{(T)}}\}$} with positive probability if $\lambda^{(T)}>1$. 

To show closeness of the eigenvalues $\lambda^{(T)}$ and $ \lambda $ we extend the operator $ M^{(T)} $ to $  \tilde{M}^{{(T)}} :\mathcal{C}_{\text{b}}(\bar{\mathbb{T}}) \to \mathcal{C}_{\text{b}}(\bar{\mathbb{T}})$ defined by 
\begin{equation}
 \tilde{M}^{{(T)}}[\psi](t') = \int_{[0,T]}\mu(t'\wedge T|t)\,\psi(t)\,\text{d}t
\end{equation}
Note that operators $ M^{(T)} $ and $ \tilde{M}^{(T)} $ share eigenvalues so we may equivalently consider the top eigenvalue \smash{$ \tilde{\lambda}^{(T)}$} of \smash{$ \tilde{M}^{(T)} $}.
Since $ \mu $ is continuous and integrable, for all $ \varepsilon>0 $ there exists~$T$ such that
\begin{align}\label{eps}
\big\| M^{}-\tilde{M}^{(T)} \big\| < \varepsilon\,,
\end{align}
where $\|\cdots\|$ is the operator norm induced by the infinity norm on $\mathcal{C}_{\text{b}}(\bar{\mathbb{T}})$.

We have already observed that the principal eigenvalue $\lambda$ of $M$ can be separated from the rest of the spectrum by a closed curve. Hence, by Kato \cite[IV $ \S $ 3.5 ]{kato2013perturbation}, we have that 
\smash{$|\lambda - \lambda^{{(T)}}|$} goes to zero with $\|M^{}-\tilde{M}^{(T)} \|$. 
In particular, if $\lambda>1$ it follows from \eqref{eps} that we can choose $ T $ such that
\begin{align}
\big| \lambda^{(T)}-\lambda\big| < \lambda-1 ,
\end{align}
and hence $\lambda^{(T)}>1$ and $\{Z_n^{_{(T)}}\}$ survives with positive probability. The untrimmed process satisfies $Z_n\geq Z_n^{_{(T)}}$ and hence also survives with positive probability.\\[-10mm]
\end{enumerate}\raggedleft$\square$
\end{proof}

To prove our main result about the survival of the secondary infection, we must explicitly identify the operator $M$, analyse its spectrum, and compute the scaling behaviour when the timescales of the infections are well separated. 

\section{Survival of the secondary infection}

\subsection{Production kernel}

The form of the kernel $ \mu(t'|t) $ may be found by considering when a type $t$ parent will have a type $t'$ offspring. For this to happen, the parent must pass on the primary infection at some time $ s $ (measured from the moment they first acquired it), and then pass on the secondary infection at time $s+t'$. The primary and secondary infections in the parent, and the primary infection in the child, must all survive long enough for this process to complete. We find it useful to break the calculation into two cases, depending on whether the primary infection is transmitted before or after the parent acquires the secondary; that is, depending on the order of $s$ and $t$. 

The case $s<t$ is illustrated in Fig.~\ref{fig:types}(i). To achieve a type $t'$ offspring in this case: the transmission time $s>0$ of the primary must occur before $t$ but after $t-t'$ (which may be negative); the secondary must be transmitted $s+t'-t$ time units after it was acquired in the parent at time $t$; the primary infection in the parent must not recover in the time between $s$ and $t$; and none of the three active infections may recover in the window of time between $t$ and $s+t'$. Putting these contributions together, we reach
\begin{align*}
\mu(t'|t, s <t)=c\int_{(t-t')_+}^{t}\big[\beta_{1}e^{-\beta_{1}s}\big]\big[\beta_{2}e^{-\beta_{2}(t'-t+s)}\big]\big[e^{-\rho_{1}(t-s)}\big]\big[e^{-(2\rho_{1}+\rho_{2})(t'-t+s)}\big]\textrm{d}s,
\end{align*}
where $(\cdots)_+$ denotes the positive part, and the prefactor of $c$ comes from the expected number of neighbours to which the infection may be transmitted. 

Similarly, the case $s\geq t$ is illustrated Fig.~\ref{fig:types}(ii). Here transmission of the primary may occur any time after $t$, with the secondary being transmitted $t'$ time units later. Both infections in the parent must survive until time $s$, after which all three infections must survive for at least $t'$ time units. The resulting expression is 
\begin{align*}
\mu(t'|t, s\geq t)=c\int_{t}^{\infty}[\beta_{1}e^{-\beta_{1}s}\big]\big[\beta_{2}e^{-\beta_{2}t'}\big]\big[e^{-(\rho_{1}+\rho_{2})(s-t)}\big]\big[e^{-(2\rho_{1}+\rho_{2})t'}\big]\textrm{d}s\,.
\end{align*}

\begin{figure}
\centering
\includegraphics[width=0.7\textwidth]{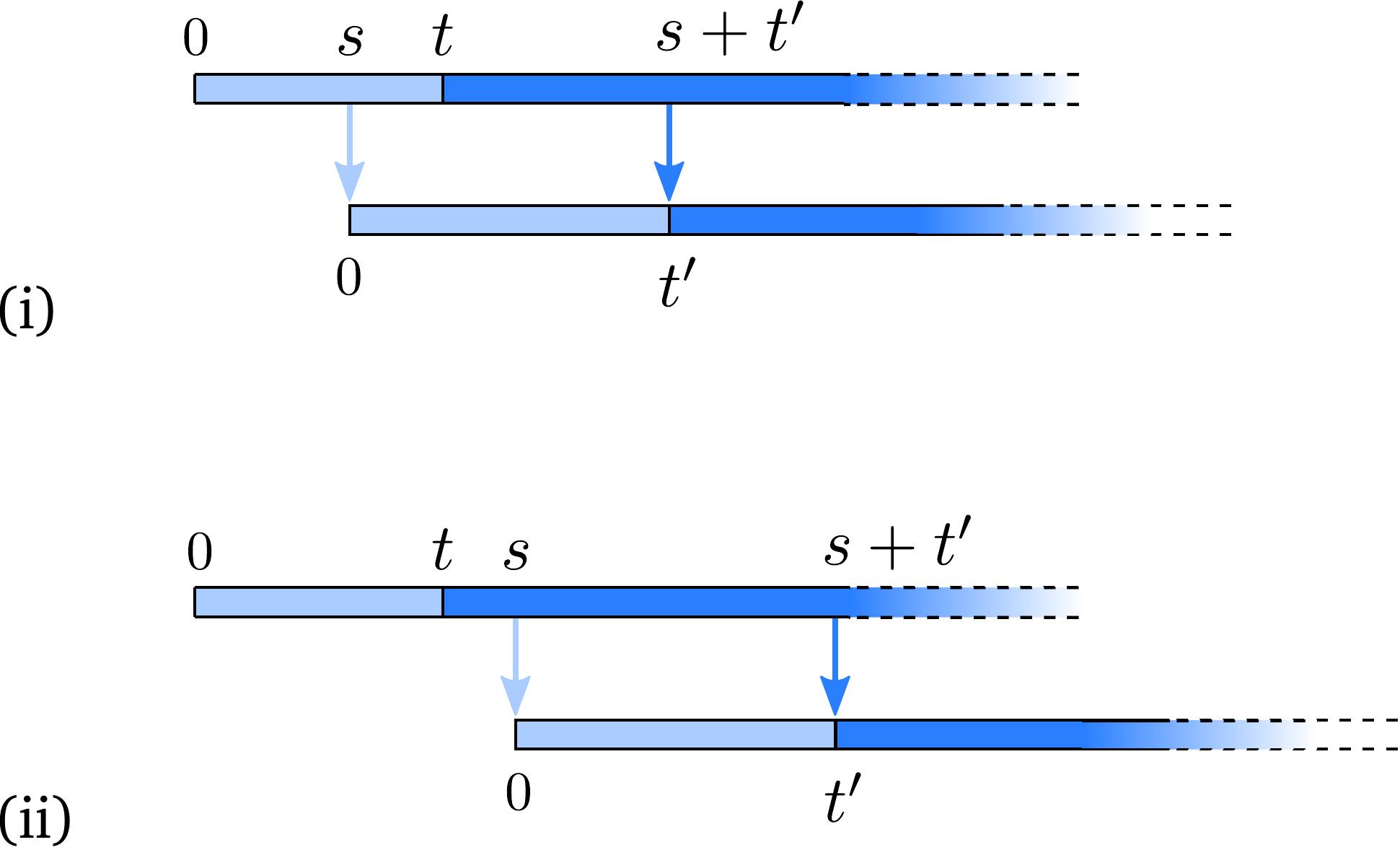}
\caption{Illustration of the timing of the necessary events for the secondary infection to successfully create a type $t'$ offspring from a type $t$ parent; in each case the top line represents the life of the parent and the bottom line that of the offspring. Pale lines denote the transmission of the primary and dark lines denote the transmission of the secondary, similarly, pale/dark regions denote the corresponding status of the nodes. We split into two cases depending on whether the time $s$ of transmission of the primary infection (measured from when it is acquired by the parent) is (i) before, or (ii) after, the time $t$ that parent acquires the secondary infection. }
\label{fig:types}
\end{figure}

Combining the two cases and evaluating the integral gives
\begin{align}
\mu(t'|t)=\left\{\begin{array}{ll}
{\frac{c\beta_{1}\beta_{2}(\beta_{2}e^{-\beta_{1}t-(\beta_{2}+2\rho_{1}+\rho_{2})t'}+(\beta_{1}+\rho_{1}+\rho_{2})e^{-\beta_{1}t-(\rho_{1}-\beta_{1})t'})}{(\beta_{1}+\rho_{1}+\rho_{2})(\beta_{1}+\beta_{2}+\rho_{1}+\rho_{2})} }& \quad\mathrm{if}\ t'\leq t\ \\
& \\
{\frac{c\beta_{1}\beta_{2}(\beta_{2}e^{-\beta_{1}t-(\beta_{2}+2\rho_{1}+\rho_{2})t'}+(\beta_{1}+\rho_{1}+\rho_{2})e^{(\beta_{2}+\rho_{1}+\rho_{2})t-(\beta_{2}+2\rho_{1}+\rho_{2})t'})}{(\beta_{1}+\rho_{1}+\rho_{2})(\beta_{1}+\beta_{2}+\rho_{1}+\rho_{2})}} & \quad\mathrm{if}\ t'>t\,.
\end{array}\right.
\label{mudef}
\end{align}

We are now ready to state our result about the spectrum of the production operator resulting from this kernel.

\begin{lemma}\label{M}
For the integral operator $M$ defined in \eqref{Mdef} with kernel $\mu$ given in \eqref{mudef}, the top eigenvalue $ \lambda $ solves the implicit equation
	\begin{align}\label{eig1}
	\frac{c\beta_{1}\beta_{2}\,{}_{0}F_{1}\left( \frac{\beta_{1}+\beta_{2}+3\rho_{1}+\rho_{2}}{\rho_{1}};-\frac{\beta_{1}\beta_{2}}{(\lambda/c)\rho_{1}^{2}}\right) }{\lambda(\beta_{1}+\rho_{1}+\rho_{2})(\beta_{1}+\beta_{2}+2\rho_{1}+\rho_{2})}={}_{0}F_{1}\left( \frac{\beta_{1}+\beta_{2}+2\rho_{1}+\rho_{2}}{\rho_{1}};-\frac{\beta_{1}\beta_{2}}{(\lambda/c)\rho_{1}^{2}}\right) 
	\end{align}
	where ${}_{0}F_{1}(a;z)=\displaystyle \sum_{k}\frac{z^{k}}{k!(a)_{k}}$ is a hypergeometric function.
\end{lemma}

\begin{proof}
From part 1 of Lemma \ref{Z}, to determine that $\lambda$ is the top eigenvalue of $M$, it is sufficient to exhibit a non-negative function $\psi$ such that $\lambda \psi=M\psi$. We begin a search for such a function by considering the successive action of $M$ starting from the initial state $\psi_0(t)=\delta_0(t)$, corresponding to a single seed infected individual who acquires the primary and secondary infections at the same instant. Defining the series 
\begin{equation}
\psi_{n+1}=M[\psi_n]\,,
\end{equation}
we observe that each iterate $\psi_n$ is a member of a family, $\Psi$, of functions that can be written as a certain positive sum of exponentials: 
\begin{equation}\label{form}
\Psi=\left\{\psi(t)=e^{-(\beta_{2}+\rho_{1}+\rho_{2})t}\sum_{k\geq 1}a_{k}e^{-k\rho_{1}t}\,:\, a_k\geq 0 \right\}\,.
\end{equation}
We look for an eigenfunction of $M$ that lies in $\Psi$. The eigenvalue equation $\lambda \psi=M[\psi]$ is thus reduced to a statement about the coefficients $\{a_k\}$. Specifically, we find
\begin{align}
\lambda \psi(t)&=\int_\mathbb{T}\mu(t|t')\psi(t')\textrm{d}t'\nonumber\\
&\Downarrow\nonumber\\
\lambda e^{-(\beta_{2}+\rho_{1}+\rho_{2})t}\sum_{k\geq 1}a_{k}e^{-k\rho_{1}t}&=c e^{-(\beta_{2}+\rho_{1}+\rho_{2})t}\sum_{k\geq 1}a_{k}(b_{k}e^{-\rho_{1}t}-d_{k}e^{-(k+1)\rho_{1}t})\,,
\label{eveq}
\end{align}
where 
\begin{align*}
b_{k}=&\frac{\beta_{1}\beta_{2}(\beta_{1}+(k+1)\rho_{1}+\rho_{2})}{k\rho_{1}(\beta_{1}+\rho_{1}+\rho_{2})(\beta_{1}+\beta_{2}+(k+1)\rho_{1}+\rho_{2})} \\
d_{k}=&\displaystyle \frac{\beta_{1}\beta_{2}}{k\rho_{1}(\beta_{1}+\beta_{2}+(k+1)\rho_{1}+\rho_{2})}\,.
\end{align*}
Equating coefficients in \eqref{eveq} determines
\begin{equation}
\lambda=c\sum_{k\geq 1}a_{k}b_{k}
\label{lameq}
\end{equation}
where the $\{a_k\}$ are found to satisfy 
\begin{equation}
a_{k+1}=-\frac{ca_{k}d_{k}}{\lambda}\,.
\end{equation}
This recursive equation specifies a solution up to a multiplicative constant:
\begin{equation}
a_{k}=\left( -\frac{\beta_{1}\beta_{2}}{(\lambda/c)\rho_{1}^{2}}\right) ^{(k-1)}\frac{a_{1}}{(k-1)!((\beta_{1}+\beta_{2}+2\rho_{1}+\rho_{2})/\rho_{1})_{k-1}}\,,
\label{ak}
\end{equation}
where $(\cdots)_k$ denotes the Pochhammer symbol. Combining this result with \eqref{lameq}, yields the implicit equation \eqref{eig1} for $\lambda$ given in the statement.

\raggedleft$\square$
\end{proof}

\subsection{Bounds on the ratio of hypergeometric functions}

Recall that the survival of the primary infection is dependent only on its birth-death ratio $\alpha$ and the connectivity of the underlying graph $c$, while the secondary infection additionally depends on its relative speed when compared primary, $\varphi:=\beta_1/\beta_2$. As per the discussion in Section \ref{branch2}, we specialise to the case that $\beta_1/\rho_1=\beta_2/\rho_2=\alpha$. Then the implicit eigenvalue equation \eqref{eig1} can be rewritten in terms of the parameters $\alpha$ and $\varphi$ to give
\begin{equation}
\frac{c}{\lambda}=\frac{(1+\alpha+\alpha\varphi)(1+\varphi+2\alpha+\alpha\varphi)}{\varphi} \frac{1}{\mathrm{\Phi}_\gamma(c\varphi/\lambda\alpha^2)}\,,
\end{equation}
where $\gamma=(1+\varphi)(1+1/\alpha)$ and $\mathrm{\Phi}$ denotes the hypergeometric ratio
\begin{equation}
\mathrm{\Phi}_a(z):=\frac{_0F_1(a+2;-z)}{_0F_1(a+1;-z)}\,.
\end{equation}
Our strategy to prove the scaling relations claimed in Theorem \ref{th1}, will be to replace this function by suitably simple upper and lower bounds with the same asymptotic behaviour. Fortunately, there is a substantial literature on topic that we may draw on.

\begin{lemma}\label{Phi}
For $a>0$ write $j_{a}$ for the smallest positive root of $J_a$, the Bessel function of the first kind. Then
\begin{equation}
a(a+2)<j_{a}^2<4(a+1)(a+2)\,,
\label{jbounds}
\end{equation}
and for all $z\in(0,j_{a})$ we have 
\begin{equation}
1<\mathrm{\Phi}_a(z)<1+\frac{4z}{j_{a}^2-4z}\,.
\label{Qbounds}
\end{equation}
\end{lemma}

\begin{proof}
Ismail and Muldoon \cite{Ismail1988} list many different bounds on $j_{a}$, including those in \eqref{jbounds} coming from formulas (6.7) and (6.22) in that article. For the second part, it is well-known \cite{Abramowitz1970} that the Bessel functions of the first kind may be expressed as
\begin{align*}
J_a(x)=\frac{(x/2)^{a}}{\Gamma(a+1)}\,{}_0F_1(a+1;-x^2/4)\,,
\end{align*}
hence, introducing $x=2\sqrt{z}$, we obtain
\begin{align}
\mathrm{\Phi}_a(z)=\frac{2(a+1)}{x}\frac{J_{a+1}(x)}{J_{a}(x)}\,.
\label{QJ}
\end{align}
This function has previously been studied by Ifantis and Siafarikas \cite{Ifantis1990}, who proved various inequalities including their formulas (1.2) and (2.17) which imply the lower and upper bounds of \eqref{Qbounds}. 

\raggedleft$\square$
\end{proof}

\subsection{Proof of Theorem 1}
\begin{proof}
As argued previously, in the limit of large Erd\H{o}s-R\'{e}nyi random graphs with mean degree $c$, the survival probability of the secondary infection coincides with that of a multi-type branching process $\{Z_n\}$ with production kernel given by equation \eqref{mudef}. From Lemma~\ref{M} and Theorem \ref{Z} we establish that $Z_n$ has non-zero probability to survive indefinitely if and only if $\lambda^\star>1$, where $\lambda^\star$ is the largest real number satisfying
\begin{align}\label{eig2}
\frac{\beta_{1}\beta_{2}{}_{0}F_{1}\left( \frac{\beta_{1}+\beta_{2}+3\rho_{1}+\rho_{2}}{\rho_{1}};-\frac{\beta_{1}\beta_{2}}{(\lambda^\star/c)\rho_{1}^{2}}\right) }{(\lambda^\star/c)(\beta_{1}+\rho_{1}+\rho_{2})(\beta_{1}+\beta_{2}+2\rho_{1}+\rho_{2})}={}_{0}F_{1}\left( \frac{\beta_{1}+\beta_{2}+2\rho_{1}+\rho_{2}}{\rho_{1}};-\frac{\beta_{1}\beta_{2}}{(\lambda^\star/c)\rho_{1}^{2}}\right)\,.
\end{align}
Noticing that $\lambda^\star$ appears only in ratio with $c$, it follows that the condition for the possibility of survival may be rewritten in terms of the critical connectivity $c^\star$ such that for $c>c^\star$ we have $\lambda^\star>1$. Rearranging equation \eqref{eig2} we straightforwardly find that $c^\star$ is the smallest positive solution to 
\begin{equation}
c^\star=\frac{(1+\alpha+\alpha\varphi)(1+\varphi+2\alpha+\alpha\varphi)}{\varphi} \frac{1}{\mathrm{\Phi}_\gamma(c^\star\varphi/\alpha^2)}\,,
\label{cs}
\end{equation}
which is precisely equation \eqref{eig0}, as required. 

To quantify the scaling behaviour of $c^\star$ for large and small $\varphi$, we recall the definition of ``big theta'' notation: $f(x)=\mathrm{\Theta}(g(x))$ as $x\to \infty$ (resp. $x\to 0$) if there exist positive constants $L$ and $U$ and $X$ such that $\forall x>X$ (resp. $x<X$) we have $$L g(x)< f(x) <U g(x)\,.$$

Two sufficient conditions are easy to check: $f(x)=\mathrm{\Theta}(g(x))$ if either
\begin{enumerate}
\item[(i)] $f(x)/g(x)$ has a positive finite limit, or
\item[(ii)] there exist functions $l(x),u(x)=\mathrm{\Theta}(g(x))$ such that $l(x)<f(x)<u(x)$.
\end{enumerate}

We will use the bounds in Lemma \ref{Phi} to exhibit functions with appropriate finite limits that sandwich $c^\star$. Specifically, recalling $\gamma=(1+\varphi)(1+1/\alpha)$, let
\begin{align}
u(\varphi)&= \frac{1}{\varphi}(1+\alpha+\alpha\varphi)(1+\varphi+2\alpha+\alpha\varphi)\,,\\
l(\varphi)&= u(\varphi)\left(1-\frac{4\varphi u(\varphi)}{\gamma(\gamma+2)\alpha^2+4\varphi u(\varphi)}\right)\,.
\label{lu}
\end{align}

First we check the upper bound. From \eqref{cs} and the lower bound of unity in equation \eqref{Qbounds} of Lemma \ref{Phi}, we have that 
\begin{equation}
c^\star=\frac{u(\varphi)}{\mathrm{\Phi}_\gamma(c^\star\varphi/\alpha^2)}<u(\varphi)\,.
\end{equation}
For the lower bound, we note that the upper bound on $\mathrm{\Phi}$ given in Lemma \ref{Phi} implies a lower bound on $c^\star$ as the smallest positive $l^\star$ satisfying the equation
\begin{equation}
l^\star=u(\varphi)\left(1+\frac{4 l^\star\varphi/\alpha^2 }{j_{\gamma}^2-4l^\star\varphi/\alpha^2}\right)^{-1}\,.
\end{equation}
In fact there is only one solution:
\begin{equation}
l^\star=u(\varphi)\left(1-\frac{4\varphi u(\varphi)}{j^2_\gamma\alpha^2+4\varphi u(\varphi)}\right)\,.
\end{equation}
The lower bound $l(\varphi)<c^\star$ given in \eqref{lu} follows immediately from this and the lower bound on $j_\gamma^2$ given in equation \eqref{jbounds} of Lemma \ref{Phi}.

It remains to check that the upper and lower bounds both have the desired scaling in large and small $\varphi$. We begin with $u(\varphi)$, which has easily determined limits
\begin{equation}
\lim_{\varphi\to0}\varphi\, u(\varphi) = (1+\alpha)(1+2\alpha)\,,\quad \lim_{\varphi\to\infty}\frac{u(\varphi)}{\varphi} = \alpha(1+\alpha)\,,
\label{uphi}
\end{equation}
both of which are finite and positive, implying $u(\varphi)=\mathrm{\Theta}(\varphi)$ for large $\varphi$ and $u(\varphi)=\mathrm{\Theta}(1/\varphi)$ for small $\varphi$. For the lower bound we use these results to obtain  
\begin{equation}
1-\frac{4\varphi u(\varphi)}{\gamma(\gamma+2)\alpha^2+4\varphi u(\varphi)}\to \frac{1+3\alpha}{1+\alpha(8\alpha^2+4\alpha+3)}\in(0,\infty) \quad\textrm{as} \quad \varphi\to0\,,
\end{equation}
and 
\begin{equation}
1-\frac{4\varphi u(\varphi)}{\gamma(\gamma+2)\alpha^2+4\varphi u(\varphi)}\to \frac{1+\alpha}{1+\alpha+4\alpha^3}\in(0,\infty) \quad\textrm{as} \quad \varphi\to\infty\,.
\end{equation}
It follows from the defintion of $l(\varphi)$ and finiteness of these limits that $l(\varphi)$ has the same scaling form as $u(\varphi)$ for both large and small arguments. Since $u$ and $l$ sandwich $c^\star$, the desired scaling is confirmed. 

\raggedleft$\square$
\end{proof}

\section{Discussion}
Theorem 1 provides an exact but implicit formula for the region in which survival of the secondary infection is possible (in the limit of infinitely large graphs), and establishes the scaling behaviour of the boundary of this region for large and small values of the parameter $\varphi=\beta_1/\beta_2$ which controls the relative timescales of the two infections. Knowledge of this scaling behaviour is enough to prove that, for fixed $\alpha$ and $c$, the survival of the secondary infection is confined to a bounded region of $\varphi$ values --- this is the reentrant phase transition of our title. Figure \ref{fig:networkvsbranching7084} shows the results of numerical simulations of both the branching process and the network model to illustrate this phenomenon. 
\begin{figure}
\centering
\includegraphics[width=0.8\linewidth]{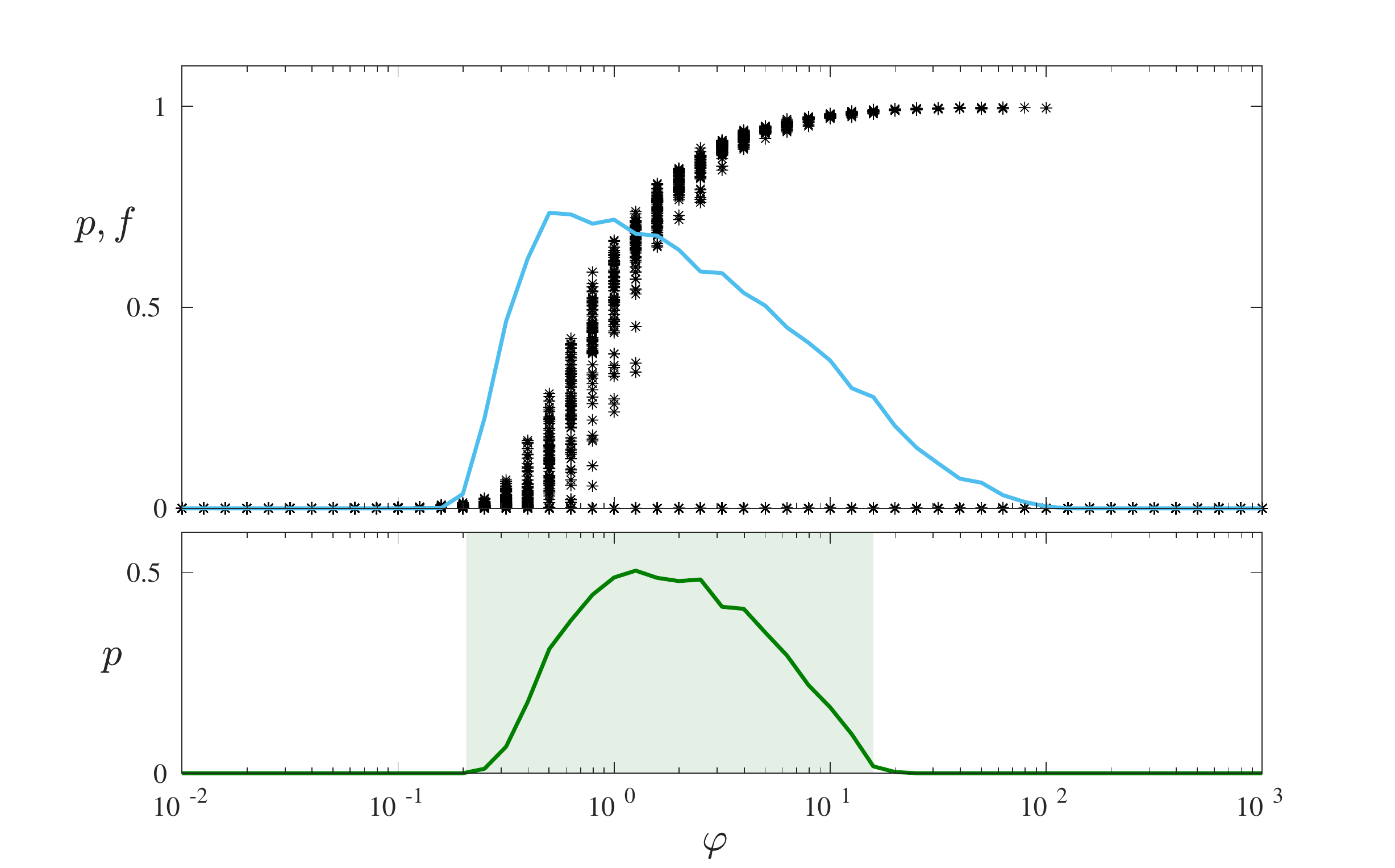}
\caption{The top panel shows the fractional size of outbreaks (f, stars) and the probability of an outbreak of size >100 (p, blue curve) \textcolor{black}{of the secondary infection,} measured from 1000 simulations of ER networks with mean degree c=10 and N=10000 nodes. The bottom panel shows on the same scale the theoretical survival region of the branching process (pale green box) and the probability of the branching process to reach size >100 (p, green curve), measured from 1000 simulations of the branching process.}
\label{fig:networkvsbranching7084}
\end{figure}

It is interesting to note that the simulations of the network process and the limiting branching process are not in perfect agreement. Viewing the mean outbreak size over 1000 runs of the model we see in Figure \ref{fig:networkvsbranching7084} that, while we have agreement with the branching process for small values of $ \varphi $, large outbreaks still seem to be possible beyond the point predicted by the branching process. Moreover, considering the individual simulation results it seems that this unexpected tail is comprised of a few  very large outbreaks; while outbreaks of any size are rare for large values of $ \varphi $, when they do happen they reach most of the graph. By considering the infection spread in a closed connected community we start to encounter finite size effects. Recall that the branching approximation is only valid when the number of infected is relatively small compared to the size of the graph. As the outbreak becomes large the approximation breaks down, a problem exacerbated by the two levels of infection we study. Furthermore in a more highly connected environment we may have the existence of transmission routes for the secondary infection to primary infected cousins as well as direct descendants allowing opportunity for the secondary infection to progress before direct primary progression. Similar finite size scaling effects have been observed in other coevolving infection models; see \cite{Cai2015a} for example. 

Comparing the average outbreak size with individual realisations demonstrates an interesting choice of risk vs reward in the strategy of a secondary infection\textcolor{black}{, due to the different locations of the maxima of the curves shown in the top panel of Figure 4. The values of $\varphi$ for which outbreaks are most likely to occur (blue curve) are in the lower end of the survival window, corresponding to smaller total outbreak sizes (black stars). Conversely, larger values of $ \varphi $ have potential for much larger outbreaks, but come with a higher risk of rapid extinction.} Looking at this another way, in nature we should expect \textcolor{black}{survival probability to be a strongly selected characteristic, and hence to find that the majority of secondary infections reach} only a minority of primary hosts. 

The work presented here could easily be extended to a host of other random graph models, for example by building on techniques of \cite{doi:10.1137/1.9781611970333,bhamidi2017,10.2307/2667184}. It may also be interesting to explore the application of the model (or variants) to other areas, including: the successive invasions of different species necessary to rebuild a diverse ecosystem in a damaged habitat; the evolution of hyperparasitism (that is, parasites that live on other parasites); radicalisation, and the incremental spread of increasingly extreme political views through social media.

\begin{figure}
	\centering
	\includegraphics[width=0.8\linewidth]{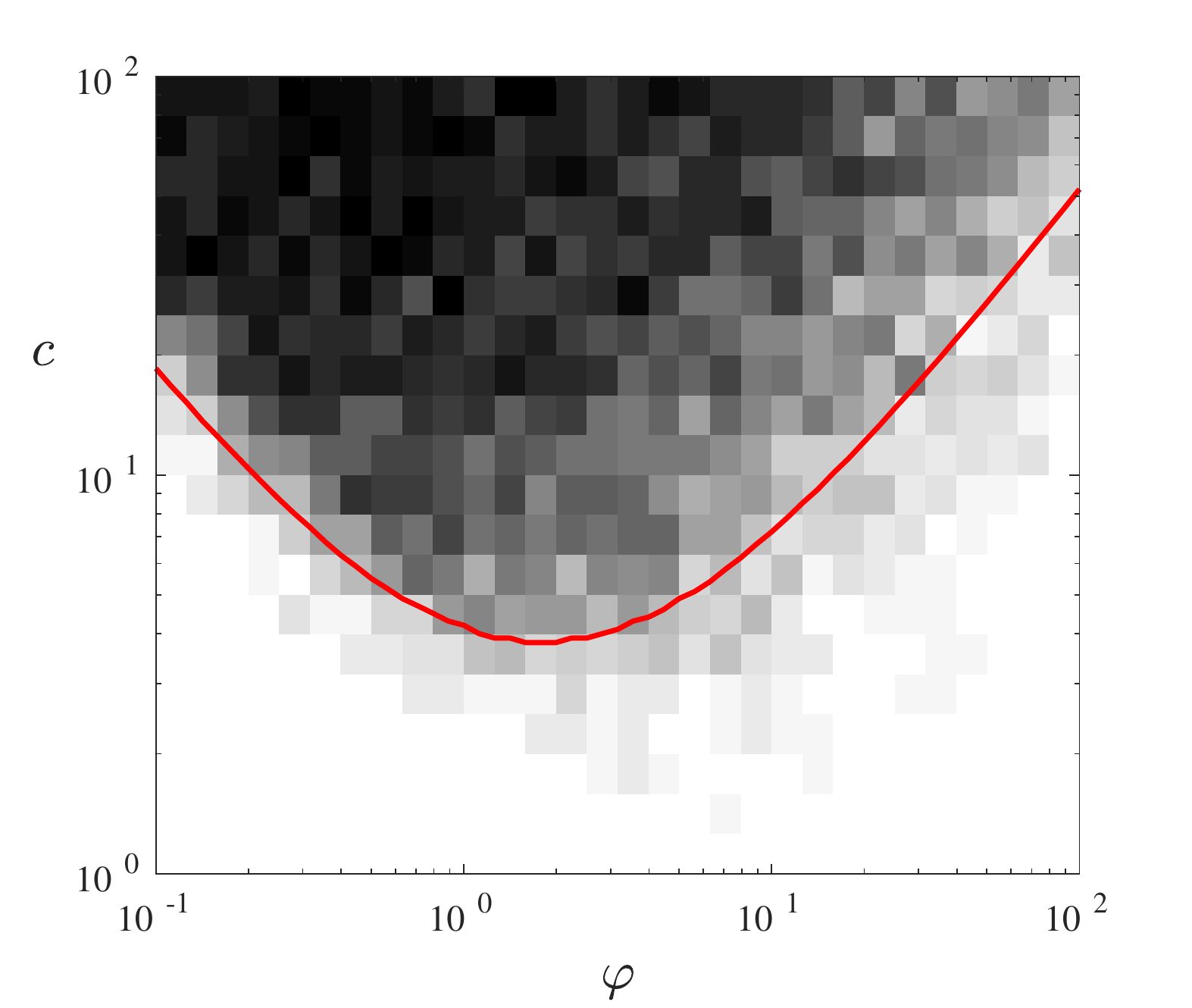}
	\caption{The density plot shows the probability (estimated as a fraction of 25 simulations per pixel) of an outbreak of size >100, starting from a single infected node, in an ER network of 10000 nodes. The red line is the boundary of the region where $\lambda>1$. }
	\label{fig:outbreakprobabilitysimn5000s257085}
\end{figure}

\newpage

\end{document}